\documentclass[12pt,a4paper]{amsart}
\usepackage[latin1]{inputenc}
\usepackage{graphicx}
\usepackage{amssymb, amsmath}
\usepackage{geometry}
\usepackage{enumerate}
\usepackage[colorlinks=true,linkcolor=blue,urlcolor=blue,citecolor=blue]{hyperref}
\newtheorem{theorem}{Theorem}[section]
\newtheorem{definition}[theorem]{Definition}
\newtheorem{lemma}[theorem]{Lemma}

\newtheorem{proposition}[theorem]{Proposition}

\theoremstyle{definition}
\newtheorem{example}[theorem]{Example}
\newtheorem{remark}[theorem]{Remark}

\begin{document}
\title[The Hardy-Littlewood maximal operator on graphs]{Best constants for the  Hardy-Littlewood maximal operator on finite graphs}
\author{Javier Soria}
\address{Department of  Applied Mathematics and Analysis, University of Barcelona, Gran Via 585, E-08007 Barcelona, Spain.}
\email{soria@ub.edu}

\author{Pedro Tradacete}
\address{Mathematics Department, Universidad Carlos III de Madrid, E-28911 Legan\'es, Madrid, Spain.}
\email{ptradace@math.uc3m.es}

\thanks{The first author has  been partially supported by the Spanish Government grants MTM2010-14946 and MTM2013-40985-P, and the Catalan Autonomous Government grant 2014SGR289. The second author has   been partially supported by the Spanish Government grants  MTM2010-14946 and MTM2012-31286, and also Grupo UCM 910346.}


\begin{abstract}
We study the behavior of averages for functions defined on finite graphs $G$, in terms of the Hardy-Littlewood maximal operator $M_G$. We explore the relationship between the geometry of a graph and its maximal operator and prove that $M_G$ completely determines $G$ (even though embedding properties for the graphs do not imply pointwise inequalities for the maximal operators). Optimal bounds for the $p$-(quasi)norm of a general graph $G$ in the range $0<p\le1$ are given, and it is shown that the complete graph $K_n$ and the star graph $S_n$ are the extremal graphs attaining, respectively, the lower and upper estimates. Finally, we study weak-type estimates and some connections with the dilation and overlapping indices of a graph.
\end{abstract}

\maketitle

\thispagestyle{empty}

\section{Introduction}

Given a  simple, connected, and finite graph $G=(V,E)$ (conditions that we will always assume from now on), where $V$ is a (finite) set of vertices and $E$ the set of edges between them, for a function $f:V\rightarrow \mathbb R$ we can consider the (centered) Hardy-Littlewood maximal operator
$$
M_G f(v)=\sup_{r\ge 0} \frac{1}{|B(v,r)|}\sum_{w\in B(v,r)} |f(w)|.
$$
Here $B(v,r)$ denotes the ball of center $v$ and radius $r$ on the graph, equipped with the metric $d_G$ induced by the edges in $E$. That is, given $v,w\in V$ the distance $d_G(v,w)$ is the number of edges in a shortest path connecting $v$ and $w$, and
$$
B(v,r)=\{w\in V(G):d_G(v,w)\le r\}.
$$
 For example, $B(v,r)=\{v\}$, if $0\le r<1$ and $B(v,r)=\{v\}\cup N_G(v)$, if $1\le r<2$, where $N_G(v)$ is the set of neighbors of $v$. Also, given a finite set $A$ we denote its cardinality by $|A|$. We will also use the  notations $n=|V|$ and $m=|E|$ (we refer to \cite{Bol, BoMu} for standard notations and definitions on graphs).

The distance $d_G$ introduced above only takes natural numbers as values and hence the radius $r>0$ considered in the definition of the Hardy-Littlewood maximal operator can be taken to be a natural number. Moreover, since the diameter of a graph of $n$ vertices is at most $n-1$, we can compute
$$
M_G f(v)=\max_{k=0,\ldots,n-1} \frac{1}{|B(v,k)|}\sum_{w\in B(v,k)} |f(w)|.
$$
This kind of averaging operators have been studied in connection with harmonic functions and the Laplace operator on trees \cite{KoPi,CoMeSe}.

Our main interest in this work focuses on finding the sharp constants $C_{G,p}$ in inequalities of the form
\begin{equation}\label{desfp}
\|M_G f\|_p\leq C_{G,p} \|f\|_p,
\end{equation}
for $0<p\leq\infty$; i.e.,
$$
C_{G,p}=\|M_G\|_p=\sup_f\frac{\|M_G f\|_p}{ \|f\|_p,},
$$
where for a function $f:V\rightarrow \mathbb R$ we denote by $\|f\|_p=\big(\sum_{v\in V}|f(v)|^p\big)^{1/p}.$ It is clear that $|f(v)|\le M_Gf(v)\le\Vert f\Vert_\infty$, and hence  $\|M_{G}\|_\infty=1$, for every graph $G$. Therefore, we only need to consider the range $0<p<\infty$.

The motivation for studying  \eqref{desfp} comes from the discretization results proved for the Hardy-Littlewood maximal function in $\mathbb R$ in terms of Dirac deltas \cite{Guz,TrSo,Mel}, which   is closely related  to the case of a linear tree $L_n$ (see Proposition~\ref{lnasin} and Remark~\ref{fertit}). We will see that a richer geometric structure on the graph gives us better estimates for the maximal operator which, in turn, characterize the graph in some extremal cases (see Theorem~\ref{knsn}). In particular, we will prove that the complete graph on $n$ vertices can be characterized in terms of the equality $\|M_G\|_1=1+{(n-1)}/{n}$; while a graph of $n$ vertices will satisfy $\|M_G\|_1=1+{(n-1)}/{2}$ precisely when $G$ is isomorphic to the star graph $S_n$.

We will also consider weak-type estimates of the form $M_G:\ell^p(V)\rightarrow \ell^{p,\infty}(V)$. These are partly motivated by \cite{NaorTao}, where several results for the Hardy-Littlewood operator in metric measure spaces are given. In particular, it is proved that for an infinite rooted  regular tree $T$, the maximal operator satisfies
$$
\|M_T\|_{\ell^1(T)\rightarrow\ell^{1,\infty}(T)}\lesssim 1,
$$
with a constant independent of the degree of $T$ \cite[Theorem 1.5]{NaorTao}. In general, computing exactly the weak-type $(1,1)$ norm of the Hardy-Littlewood operator in a metric space $\|M\|_{L_1(X)\rightarrow L_{1,\infty}(X)}$ is a hard problem. In ultrametric spaces, this norm equals one, while for the real line $\mathbb R$, Melas showed \cite{Mel} it equals $(11+\sqrt{61})/12$. Optimal bounds in $L^p(\mathbb R)$, for the uncentered maximal function, are proved in \cite{GraMo}.

Other results involving maximal operators on infinite graphs can be found in \cite{BdMar}. In \cite{arbardomsor} boundedness of some Hardy type averaging operators were also considered in the setting of partially ordered measure spaces, which include the case of infinite trees.

In our analysis of weak-type estimates we will introduce two indices associated with coverings of a graph: the dilation and the overlapping indices. These will provide an upper bound for the weak-type $(1,1)$ estimate of $M_G$ (Theorem~\ref{indices}).

Recall that two graphs $G_1$, $G_2$ are said to be isomorphic if there is a permutation of the vertices $\pi:V\rightarrow V$ such that $v,w\in V$ are the endpoints of an edge in $E_{G_1}$ if and only if $\pi(v)$ and $\pi(w)$ are the endpoints of an edge in $E_{G_2}$. In this case, we will write $G_1\sim G_2$. It is clear that if $G_1\sim G_2$, then $M_{G_2}f(\pi(v))=M_{G_1}f(v)$ and hence $\| M_{G_1}\|_p=\| M_{G_2}\|_p$, $0<p\le\infty$. That the converse is not true can be seen in Example~\ref{examples4}.

Given a graph $G$, the degree of a vertex $v\in V_G$, denoted by $d_G(v)$, is the number of edges in $E_G$ which have $v$ as one of the endpoints; that is, $d_G(v)=|N_G(v)|$. For $j\in V$ we will consider the Kronecker delta
\begin{equation}\label{krodel}
\delta_j(i)=
\begin{cases}
  1, & \textrm{for }\,i=j,  \\[.2cm]
  0, &   \textrm{for }\,i\not= j.
\end{cases}
\end{equation}
We will use the notation $A\lesssim B$, whenever there exists $C>0$ (independent of the main parameters involved, like the dimension $n\in\mathbb N$ or $0<p<\infty$) such that $A\le CB$. Similarly for $A\gtrsim B$. As usual, $A\approx B$ means that $A\lesssim B$ and $A\gtrsim B$.
\medskip

The paper is organized as follows: In Section~\ref{propbest} we prove in Theorem~\ref{eqmaxgr} that $M_G$ completely determines $G$. Lemma~\ref{normwithdeltas} is our main tool to easily calculate the norm of the maximal operator $M_G$ on the range $0<p\le 1$ and, in particular, we consider the case of the complete graph $K_n$. We finish by introducing, in Proposition~\ref{estresttyp},  some estimates of restricted type. In Section~\ref{optgr} we show in Theorem~\ref{knsn} that $K_n$ and the star graph $S_n$ are optimal cases for the boundedness of $M_G$ in $\ell^p(G)$, $0<p\le 1$, and get also sharp estimates for the linear graph $L_n$. To complete the information for the strong-type estimates, we calculate the norm for the star  graph, on the range $1<p<\infty$. Finally, in Section~\ref{weaktes}, we consider the study of weak-type estimates on $0<p<\infty$, and establish in Theorem~\ref{indices} a relationship, for $p=1$, with some geometrical indices associated to the graph.

\section{General properties and best constants}\label{propbest}

Let $K_n$ denote the complete graph with $n\ge 2$ vertices, which we are going to label as $V=\{1,\ldots, n\}$. As a metric space, this is the simplest among all graphs with $n$ vertices, since given any $v\in V$ we have
$$
B(j,r)=
          \begin{cases}
            \{j\} , & \textrm{for } 0\le r<1, \\[.2cm]
            V , & \textrm{for }r\geq1.
          \end{cases}
$$
Therefore, the maximal operator takes the form
\begin{equation}\label{maoman}
M_{K_n} f(j)=\max\Big\{|f(j)|,\frac1n\sum_{k\in V}|f(k)|\Big\}.
\end{equation}
The operator $M_{K_n}$ is the smallest, in the pointwise ordering, among all $M_G$,  with $G$ a graph of $n$ vertices. That is,  for every positive function $f:V\rightarrow \mathbb R$ and every $j\in V$, we have that
\begin{equation}\label{desphl}
M_{K_n}f(j)\leq M_Gf(j).
\end{equation}
In particular, if $0< p\leq \infty$ and  $G$ is a graph with $n$ vertices, then
$$
\|M_{K_n}\|_p\leq\|M_G\|_p.
$$

\begin{remark}\label{notincl}
Regarding \eqref{desphl}, it is worth mentioning that, in general, it is not true that if $G_1\subset G_2$ (i.e., $V(G_1)=V(G_2)$ and $E(G_1)\subset E(G_2)$), then $M_{G_2}f\le M_{G_1}f$. For example, if $V=\{1,2,3,4\}$, $G_1$ is a linear tree with leafs 1 and 4,  $G_2$  is the 4-cycle $C_4$ (with a clockwise orientation of $V$), and $f=\delta_4$ is the Kronecker delta (see \eqref{krodel} for the definition), then $G_1\subset G_2$, but it is easy to prove that, however,   $M_{G_2}\delta_4(1)=1/3>1/4=M_{G_1}\delta_4(1)$.
\end{remark}

Contrary to the minimality property \eqref{desphl}  of the complete graph $K_n$, there is no graph $G$ whose maximal operator $M_G$ is the largest in the pointwise ordering among all graphs with $n\ge3$ vertices ($n=2$ is trivial since $K_2$ is the only example). That is, there exists no graph $G_{\rm max}$ such that, for every graph $G$ with $V(G)=V(G_{\rm max})$, and every  function $f:V\rightarrow \mathbb R$ we have
$$
M_G f(j)\leq M_{G_{\rm max}} f(j),\quad \textrm{for each } j\in V.
$$
However, we will prove in Theorem~\ref{knsn} that, in terms of the (quasi)norm $\Vert M_G\Vert_p$, for $0<p\le 1$, we do have the existence of a maximal graph (namely, the star $S_n$).

\begin{proposition}
If $G$ is a graph with $n\ge3$ vertices, then there exists $j\in V(G)$,  $f:V\rightarrow \mathbb R^+$, and another graph $G'$, with $V(G')=V(G)$, so that $M_Gf(j)<M_{G'}f(j)$.
\end{proposition}

\begin{proof}
Since $G$ has at least 3 vertices, then there is a vertex $j\in V$  with degree $d_G(j)\geq2$. Let $k\in V$ be a neighbor of $j$. Let $G'=G_{j,k}$ be a linear tree with $n$ vertices and such that $j$ has degree 1 (it is a leaf in $G'$), and $k$ is the only neighbor of $j$ in $G'$. Let us consider the function $f(j)=1/3$, $f(k)=2/3$, and $f(l)=0$ elsewhere. Then, 
$$
M_Gf(j)=\max\{1/3, 1/(d_G(j)+1)\}=1/3\text{ and } M_{G'}f(j)=\max\{1/3, 1/2\}=1/2. 
$$
Hence,
$
M_G f(j)<M_{G'}f(j).
$
\end{proof}

We can also consider a maximal operator involving the averages for all isomorphic graphs to a given one. That is, given $G$, for $f:V\rightarrow \mathbb R$ and $j\in V$, we define:
$$
M_{[G]}f(j)=\max_{H\sim G}M_Hf(j).
$$
For this larger operator, we can actually prove the following optimal pointwise estimates (see Proposition~\ref{optlnsn} for further properties):

\begin{proposition}\label{maxec}
Let $L_n$ be a linear tree. Then, for every graph with $n$ vertices and any function $f:\{1,\dots,n\}\rightarrow \mathbb R^+$,
$$
M_{[G]}f(j)\le M_{[L_n]}f(j), \quad j\in\{1,\dots,n\}.
$$
\end{proposition}

\begin{proof}
Given $G$ and $j\in\{1,\dots,n\}$, it is easy to find a linear graph $L_n$ so that, for every $0\le r\le n-1$, there exists $0\le s(r)\le n-1$ such that
$$
B_G(j,r)=B_{L_n}(j,s(r)),
$$
and $j$ is a leaf of $L_n$. For example, if $0\le r<1$, take $s(r)=r$ and $B_G(j,r)=\{j\}=B_{L_n}(j,r)$. If $1\le r<2$, then we order $N_G(j)$ in $L_n$ to obtain that $B_{L_n}(j,d_G(j))= \{j\}\cup N_G(j)=B_G(j,r)$ (i.e.; $s(r)=d_G(j)$), and so on (see Figure~{1}). Then, for every $f:\{1,\dots,n\}\rightarrow \mathbb R^+$,

\begin{align*}
M_Gf(j)&=\max\bigg\{\displaystyle\frac{\sum_{k\in B_G(j,r)}f(k)}{|B_G(j,r)|}:0\le r\le n-1\bigg\}\\
&\le \max\bigg\{\displaystyle\frac{\sum_{k\in B_{L_n}(j,s)}f(k)}{|B_{L_n}(j,s)|}:0\le s\le n-1\bigg\}\\
&=M_{L_n}f(j)\le M_{[L_n]}f(j),
\end{align*}

\noindent
which gives $M_{[G]}f(j)\le M_{[L_n]}f(j)$, for every $j\in\{1,\dots,n\}. $
\end{proof}

\medskip

   \begin{center}
   %Created by jPicEdt 1.x
%Standard LaTeX format (emulated lines)
%Fri Jan 31 17:01:29 CET 2014
\unitlength 1mm
\begin{picture}(151.25,25.00)(0,0)

\linethickness{0.15mm}
%Ellipse 0 0(0.63,14.38)(1.25)(1.25) blacken arcStart=0.0 arcExtent=0.0
\put(0.63,14.38){\circle*{1.25}}
%End Ellipse

\linethickness{0.15mm}
%Ellipse 0 0(10.63,14.38)(1.25)(1.25) blacken arcStart=0.0 arcExtent=0.0
\put(10.63,14.38){\circle*{1.25}}
%End Ellipse

\linethickness{0.15mm}
%Ellipse 0 0(20.63,24.38)(1.25)(1.25) blacken arcStart=0.0 arcExtent=0.0
\put(20.63,24.38){\circle*{1.25}}
%End Ellipse

\linethickness{0.15mm}
%Ellipse 0 0(20.63,14.38)(1.25)(1.25) blacken arcStart=0.0 arcExtent=0.0
\put(20.63,14.38){\circle*{1.25}}
%End Ellipse

\linethickness{0.15mm}
%Ellipse 0 0(30.00,14.38)(1.25)(1.25) blacken arcStart=0.0 arcExtent=0.0
\put(30.00,14.38){\circle*{1.25}}
%End Ellipse

\linethickness{0.15mm}
%Ellipse 0 0(40.63,24.38)(1.25)(1.25) blacken arcStart=0.0 arcExtent=0.0
\put(40.63,24.38){\circle*{1.25}}
%End Ellipse

\linethickness{0.15mm}
%Ellipse 0 0(40.63,14.38)(1.25)(1.25) blacken arcStart=0.0 arcExtent=0.0
\put(40.63,14.38){\circle*{1.25}}
%End Ellipse

\linethickness{0.15mm}
%Ellipse 0 0(50.63,14.38)(1.25)(1.25) blacken arcStart=0.0 arcExtent=0.0
\put(50.63,14.38){\circle*{1.25}}
%End Ellipse

\linethickness{0.15mm}
%Ellipse 0 0(60.63,14.38)(1.25)(1.25) blacken arcStart=0.0 arcExtent=0.0
\put(60.63,14.38){\circle*{1.25}}
%End Ellipse

\linethickness{0.15mm}
%Ellipse 0 0(70.63,14.38)(1.25)(1.25) blacken arcStart=0.0 arcExtent=0.0
\put(70.63,14.38){\circle*{1.25}}
%End Ellipse

\linethickness{0.15mm}
%Ellipse 0 0(80.63,14.38)(1.25)(1.25) blacken arcStart=0.0 arcExtent=0.0
\put(80.63,14.38){\circle*{1.25}}
%End Ellipse

\linethickness{0.15mm}
%Ellipse 0 0(90.63,14.38)(1.25)(1.25) blacken arcStart=0.0 arcExtent=0.0
\put(90.63,14.38){\circle*{1.25}}
%End Ellipse

\linethickness{0.15mm}
%Ellipse 0 0(100.63,14.38)(1.25)(1.25) blacken arcStart=0.0 arcExtent=0.0
\put(100.63,14.38){\circle*{1.25}}
%End Ellipse

\linethickness{0.15mm}
%Ellipse 0 0(110.63,14.38)(1.25)(1.25) blacken arcStart=0.0 arcExtent=0.0
\put(110.63,14.38){\circle*{1.25}}
%End Ellipse

\linethickness{0.15mm}
%Ellipse 0 0(120.63,14.38)(1.25)(1.25) blacken arcStart=0.0 arcExtent=0.0
\put(120.63,14.38){\circle*{1.25}}
%End Ellipse

\linethickness{0.15mm}
%Ellipse 0 0(130.63,14.38)(1.25)(1.25) blacken arcStart=0.0 arcExtent=0.0
\put(130.63,14.38){\circle*{1.25}}
%End Ellipse

\linethickness{0.15mm}
%Ellipse 0 0(140.63,14.38)(1.25)(1.25) blacken arcStart=0.0 arcExtent=0.0
\put(140.63,14.38){\circle*{1.25}}
%End Ellipse

\linethickness{0.15mm}
%Ellipse 0 0(150.63,14.38)(1.25)(1.25) blacken arcStart=0.0 arcExtent=0.0
\put(150.63,14.38){\circle*{1.25}}
%End Ellipse

\linethickness{0.15mm}
%Polygon 0 0(70.63,14.38)(150.63,14.38) blacken
\put(70.63,14.38){\line(1,0){80.00}}
%End Polygon

\linethickness{0.15mm}
%Polygon 0 0(0.63,14.38)(60.63,14.38) blacken
\put(0.63,14.38){\line(1,0){60.00}}
%End Polygon

\linethickness{0.15mm}
%Polygon 0 0(10.63,14.38)(20.63,24.38) blacken
\multiput(10.63,14.38)(0.12,0.12){83}{\line(1,0){0.12}}
%End Polygon

\linethickness{0.15mm}
%Polygon 0 0(20.63,24.38)(20.63,14.38) blacken
\put(20.63,14.38){\line(0,1){10.00}}
%End Polygon

\linethickness{0.15mm}
%Polygon 0 0(20.63,24.38)(30.00,14.38) blacken
\multiput(20.63,24.38)(0.12,-0.13){78}{\line(0,-1){0.13}}
%End Polygon

\linethickness{0.15mm}
%Polygon 0 0(30.00,14.38)(40.63,24.38) blacken
\multiput(30.00,14.38)(0.13,0.12){83}{\line(1,0){0.13}}
%End Polygon

\linethickness{0.15mm}
%Polygon 0 0(40.63,24.38)(40.63,14.38) blacken
\put(40.63,14.38){\line(0,1){10.00}}
%End Polygon

\put(0.63,11.22){\makebox(0,0)[cc]{1}}

\put(10.63,11.25){\makebox(0,0)[cc]{2}}

\put(20.63,11.25){\makebox(0,0)[cc]{4}}

\put(16.88,24.38){\makebox(0,0)[cc]{3}}

\put(30.00,11.25){\makebox(0,0)[cc]{5}}

\put(36.88,24.38){\makebox(0,0)[cc]{6}}

\put(40.63,11.25){\makebox(0,0)[cc]{7}}

\put(50.63,11.25){\makebox(0,0)[cc]{8}}

\put(60.63,11.25){\makebox(0,0)[cc]{9}}

\put(70.63,11.25){\makebox(0,0)[cc]{5}}

\put(80.63,11.25){\makebox(0,0)[cc]{3}}

\put(90.63,11.25){\makebox(0,0)[cc]{4}}

\put(100.63,11.25){\makebox(0,0)[cc]{6}}

\put(110.63,11.25){\makebox(0,0)[cc]{7}}

\put(120.63,11.25){\makebox(0,0)[cc]{2}}

\put(130.63,11.25){\makebox(0,0)[cc]{8}}

\put(140.63,11.25){\makebox(0,0)[cc]{1}}

\put(150.63,11.25){\makebox(0,0)[cc]{9}}

\linethickness{0.15mm}
%Polygon 0 1(70.63,8.75)(110.63,8.75) blacken
\put(70.63,8 ){\line(1,0){40.00}}
\put(110.63,8 ){\vector(1,0){0.12}}
%End Polygon

\put(86.88,5.3){\makebox(0,0)[cc]{$B_G(5,1)$}}

\put(70.63,19.38){\makebox(0,0)[cc]{}}

\linethickness{0.15mm}
%Polygon 0 1(70.63,18.75)(130.63,18.75) blacken
\put(70.63,18.75){\line(1,0){60.00}}
\put(130.63,18.75){\vector(1,0){0.12}}
%End Polygon

\put(98.13,21.3){\makebox(0,0)[cc]{$B_G(5,2)$}}

\linethickness{0.15mm}
%Polygon 0 1(70.63,2.50)(150.63,2.50) blacken
\put(70.63,2.50){\line(1,0){80.00}}
\put(150.63,2.50){\vector(1,0){0.12}}
%End Polygon

\put(105.63,-0.50){\makebox(0,0)[cc]{$B_G(5,3)$}}

\end{picture}
 \\[.3cm]
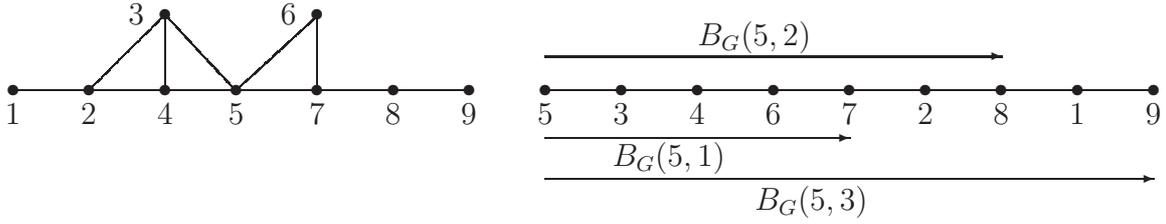
   Figure 1: A graph $G$ and its corresponding linear tree for $j=5$.
\end{center}

\medskip

We now study the relationship between the geometry of a graph and its maximal operator and prove that $M_G$ completely determines $G$, even though embedding properties for the graphs do not imply pointwise inequalities for the maximal operators (see Remark~\ref{notincl}).

\begin{theorem}\label{eqmaxgr}
Let $G_1$ and $G_2$ be two graphs with $V(G_1)=V(G_2)=\{1,\dots,n\}$. The following are equivalent:
\begin{enumerate}
\item[(i)] $G_1=G_2$.
\item[(ii)] For every $f:\{1,\dots,n\}\rightarrow \mathbb R$, $M_{G_1}f= M_{G_2}f$.
\item[(iii)] For every $k\in V$, $M_{G_1}\delta_k= M_{G_2}\delta_k$.
\end{enumerate}
\end{theorem}

\begin{proof}
The implications $(i)\Rightarrow(ii)\Rightarrow(iii)$ are clear. Let us prove $(iii)\Rightarrow (i)$: For each $k\in\{1,\dots,n\}$, we have that
\begin{equation}\label{eqmaxdel}
M_{G_1}\delta_k(j)=\frac1{|B_{G_1}(j,d_{G_1}(j,k))|}=M_{G_2}\delta_k(j)=\frac1{|B_{G_2}(j,d_{G_2}(j,k))|}.
\end{equation}
To prove that $G_1=G_2$, it suffices to show that $N_{G_1}(j)=N_{G_2}(j)$, for every vertex $j\in\{1,\dots,n\}$. Assume that $|N_{G_1}(j)|=r$, with $r=d_{G_1}(j)$ and choose an ordering of $N_{G_1}(j)=\{v_1,\dots,v_r\}$ in such a way that $d_{G_2}(j,v_1)\le d_{G_2}(j,v_2)\le \dots\le d_{G_2}(j,v_r)$.
Since $B_{G_2}(j,d_{G_2}(j,v_1))\subset B_{G_2}(j,d_{G_2}(j,v_r))$ and, using \eqref{eqmaxdel}, we also have that, for every $l\in\{1,\dots,r\}$,
$$
1+r=|B_{G_1}(j,1)|=|B_{G_1}(j,d_{G_1}(j,v_l))|=|B_{G_2}(j,d_{G_2}(j,v_l))|.
$$
Thus,  for every $l\in\{1,\dots,r\},$
\begin{equation}\label{eqballs}
B_{G_2}(j,d_{G_2}(j,v_1))=B_{G_2}(j,d_{G_2}(j,v_l)),
\end{equation}
which implies that $d_{G_2}(j,v_1)=d_{G_2}(j,v_l)$. In fact, if $d_{G_2}(j,v_1)<d_{G_2}(j,v_l)$, for some $v_l$, then $v_l\in B_{G_2}(j,d_{G_2}(j,v_l))\setminus B_{G_2}(j,d_{G_2}(j,v_1)),$ which contradicts \eqref{eqballs}.
\medskip

Finally, let us see that $d_{G_2}(j,v_1)=1$, and hence $d_{G_2}(j,v_l)=1$, for every $v_l$: If $d_{G_2}(j,v_1)>1$,   then $ B_{G_2}(j,d_{G_2}(j,v_1))$ contains a vertex $u\in N_{G_2}(j)$, which necessarily satisfies that $u\notin\{j\}\cup\{v_1,\dots,v_r\}$ (we can always find a geodesic $L$ from $j$ to $v_1$ of the form $L\equiv j,u,\cdots, v_1$). Thus,
$$
1+r=| B_{G_2}(j,d_{G_2}(j,v_1))|\ge2+r,
$$
which is a contradiction.
\medskip

Therefore, we obtain that $N_{G_1}(j)\subset N_{G_2}(j)$. Reversing the role of $G_1$ and $G_2$, or using that
$$
|N_{G_1}(j)|=|B_{G_1}(j,d_{G_1}(j,v_1))|-1=|B_{G_2}(j,d_{G_2}(j,v_1))|-1=| N_{G_2}(j)|,
$$
we conclude that $N_{G_1}(j)= N_{G_2}(j)$.
\end{proof}

\medskip

A starting point for our analysis of the norm of $\|M_G\|_p$, for $0<p\leq1$, is the following useful result. It is worth mentioning that  this estimate  is the discrete equivalent version in $\ell^p$ of  \cite[Theorem~3]{TrSo}.

\begin{lemma}\label{normwithdeltas}
Let $G$ be  graph with $n$ vertices,   and $T:\ell^p(G)\rightarrow\ell^p(G)$ be a sublinear operator, with $0<p\leq1$. Then,
$$
\|T\|_p=\max_{k\in V}\|T\delta_k\|_p.
$$
In particular, $\|M_G\|_p=\displaystyle\max_{k\in V}\|M_G\delta_k\|_p$ and $\|M_{[G]}\|_p=\displaystyle\max_{k\in V}\|M_{[G]}\delta_k\|_p$.
\end{lemma}

\begin{proof}
Since for any $0<p\le 1$ and $k\in V$, we have that $\|\delta_k\|_p=1$, then    $\|T\|_p\geq\displaystyle\max_{k\in V}\|T\delta_k\|_p$. For the converse, let $f:V\rightarrow \mathbb R$, with $\|f\|_p\leq1$; that is,
$$
f=\sum_{k\in V} a_k \delta_k,
$$
with
$\displaystyle\sum_{k\in V}|a_k|^p\leq1$. Using Holder's inequality for $0<p\leq1$, it follows that
\begin{align*}
\|T f\|_p^p&=\sum_{j\in V}|Tf(j)|^p=\sum_{j\in V}\Big|T\Big(\sum_{k\in V} a_k \delta_k\Big)(j)\Big|^p\\
&\leq \sum_{j\in V}\Big|\sum_{k\in V} |a_k| T\delta_k(j)\Big|^p \leq\sum_{j\in V}\sum_{k\in V} |a_k T\delta_k(j)|^p\\
&= \sum_{k\in V} |a_k|^p \sum_{j\in V}|T\delta_k(j)|^p=\sum_{k\in V} |a_k|^p\|T\delta_k\|_p^p\\
&\leq\max_{k\in V} \|T\delta_k\|_p^p.
\end{align*}
\end{proof}

\begin{example}\label{examples4} As an application of Lemma~\ref{normwithdeltas}, let us find $\|M_G\|_1$ for all six graphs $G$ with 4 vertices:
\begin{enumerate}[{\it(i)}]
\item $L_4$ (two vertices $\{1,4\}$ of degree 1, two vertices $\{2,3\}$ of degree 2): $\|M_{L_4}\|_1=13/6$.
\medskip

\noindent
For the vertices 1 and 2 we have
$$
M_{L_4}\delta_1(j)=
\begin{cases}
 1, &    j=1,  \\
 1/3,     & j=2,   \\
 1/4,     & j=3,4 ,
\end{cases}
\quad\textrm{and}\quad
M_{L_4}\delta_2(j)=
\begin{cases}
 1/2 ,   & j=1,  \\
 1 ,  & j=2,     \\
 1/3 ,   & j=3,4.    \\
\end{cases}
$$
Hence, $\|M_{L_4}\delta_1\|_1=11/6$ and $ \|M_{L_4}\delta_2\|_1=13/6$. By symmetry, we also have the estimates for the remaining  vertices: $\|M_{L_4}\delta_4\|_1=11/6$ and $\|M_{L_4}\delta_3\|_1=13/6$. Hence, $\|M_{L_4}\|_1=13/6$.

\

\item $C_4$ (all four vertices of degree 2): $\|M_{C_4}\|_1=23/12$.
\medskip

\noindent
Since every vertex has the same degree, we have for $k=1,\ldots,4:$
$$
M_{C_4}\delta_k(j)=
\begin{cases}
 1,   & j=k,  \\
 1/3 ,   & j\equiv k-1,k+1\,(\textrm{mod}\,4),    \\
 1/4 ,   & j\equiv k+2\,(\textrm{mod}\,4).   \\
\end{cases}
$$
Hence, $\|M_{C_4}\|_1=\|M_{C_4}\delta_k\|_1=23/12.$

\

\item $S_4$ (one vertex $\{1\}$ of degree 3, three vertices $\{2,3,4\}$ of degree 1): $\|M_{S_4}\|_1=5/2$.
$$
M_{S_4}\delta_1(j)=
\begin{cases}
 1, &    j=1,  \\
 1/2,     & j=2,3,4,
\end{cases}
\quad\textrm{and}\quad
M_{S_4}\delta_2(j)=
\begin{cases}
 1/4 ,   & j=1,3,4,  \\
 1 ,  & j=2.
\end{cases}
$$
Hence, $\|M_{S_4}\delta_1\|_1=5/2$ and $ \|M_{S_4}\delta_2\|_1=\|M_{S_4}\delta_3\|_1=\|M_{S_4}\delta_4\|_1=7/4$. Hence, $\|M_{S_4}\|_1=5/2$ (see Theorem \ref{knsn} for further information).

\

\item $K_4$ (all four vertices of degree 3): $\|M_{K_4}\|_1=7/4$.
\medskip

\noindent
This is a trivial calculation and it also follows from Theorem \ref{knsn}, with  $n=4$.

\

\item $D_4$ (two vertices $\{2,4\}$ of degree 3, two vertices $\{1,3\}$ of degree 2): $\|M_{D_4}\|_1=23/12$.
$$
M_{D_4}\delta_1(j)=
\begin{cases}
 1 ,   & j=1,  \\
 1/4 ,   & j=2,3,4,
\end{cases}
\quad\textrm{and}\quad
M_{D_4}\delta_2(j)=
\begin{cases}
 1/3 ,   & j=1,3,  \\
 1 ,   & j=2,     \\
 1/4 ,   & j=4.    \\
\end{cases}
$$
Thus, $\|M_{D_4}\delta_1\|_1=7/4, \|M_{D_4}\delta_2\|_1=23/12$. As before, by symmetry, we also have the estimates $\|M_{D_4}\delta_3\|_1=7/4, \|M_{D_4}\delta_4\|_1=23/12$. Hence, we finally obtain that $\|M_{D_4}\|_1=23/12$.

\

\item $P_4$ (one vertex $\{1\}$ of degree 1, one vertex $\{2\}$ of degree 3, two vertices $\{3,4\}$ of degree 2): $\|M_{P_4}\|_1=13/6$.
\medskip

\begin{align*}
M_{P_4}\delta_1(j)=
\begin{cases}
 1,  & j=1,  \\
 1/4 , & j=2,3,4,
\end{cases}&
\qquad\qquad\qquad
M_{P_4}\delta_2(j)=
\begin{cases}
 1/2 ,   & j=1,  \\
 1 ,   & j=2,     \\
 1/3 ,   & j=3,4,    \\
\end{cases}\\[.2cm]
 M_{P_4}\delta_3(j)&=
\begin{cases}
 1/4 ,   & j=1,2,  \\
 1 ,   & j=3,     \\
 1/3 ,   & j=4.    \\
\end{cases}
\end{align*}
Hence, $\|M_{P_4}\delta_1\|_1=7/4,\, \|M_{P_4}\delta_2\|_1=13/6,\, \|M_{P_4}\delta_3\|_1=\|M_{P_4}\delta_4\|_1=11/6.$ Thus, $\|M_{P_4}\|_1=13/6.$
\end{enumerate}
\medskip
\end{example}

In the following diagrams we exhibit the different inclusions between all (connected) graphs with 4 vertices and the order relation among the norms of the corresponding maximal operators.

$$
\left.\begin{array}{ccccc} &  & K_4 &  &  \\ &  & \cup &  &  \\P_4 & \subset & D_4 & \supset & C_4 \\\cup &  &  &  & \cup \\S_4 &  &  &  & L_4\end{array}\right.
\qquad\qquad\qquad
\left.\begin{array}{ccccc} &  & \|M_{K_4}\|_1 &  &  \\ &  & \wedge &  &  \\ \|M_{P_4}\|_1 & > & \|M_{D_4}\|_1 & = & \|M_{C_4}\|_1 \\ \wedge &  &  &  & \wedge \\ \|M_{S_4}\|_1 &  &  &  & \|M_{L_4}\|_1\end{array}\right.
$$

In particular, these examples show that we may have  non-isomorphic graphs with equal norms ($\|M_{C_4}\|_1=\|M_{D_4}\|_1$ and $\|M_{L_4}\|_1=\|M_{P_4}\|_1$). 

\medskip

The diagram however motivates the following question: Given two graphs $G_1\subset G_2$  with $n$ vertices   (in the sense that every edge in $G_1$ is an edge in $G_2$), is it always true that $\|M_{G_2}\|_1\leq\|M_{G_1}\|_1$? Recall that $G_1\subset G_2$ does not imply, in general, the pointwise inequality $M_{G_2}f\le M_{G_1}f$ (see Remark~\ref{notincl}).
\medskip

We are now going to study some optimal constants, and other estimates, for $\|M_{K_n}\|_p$. In  Section~\ref{optgr} we will see that, for $0<p\le1$, they are in fact uniquely determined by $K_n$.

\bigbreak

\begin{proposition}\label{estimateKn}\

\begin{enumerate}[(i)]

\item If $0<p\le 1$, then
$$
\|M_{K_n}\|_p=\Big(1+\frac{n-1}{n^p}\Big)^{1/p}.
$$

\item If $1< p<\infty$, then
\begin{equation}\label{ubmkp}
\Big(1+\frac{n-1}{n^p}\Big)^{1/p}\leq\|M_{K_n}\|_p\leq\Big(1+\frac{n-1}n\Big)^{1/p}.
\end{equation}
In particular, $\|M_{K_n}\|_p\approx 1$.
\end{enumerate}
\end{proposition}

\begin{proof}
Using \eqref{maoman} we have that the norm of $M_{K_n}$ can be computed as
\begin{equation}\label{normnp}
\|M_{K_n}\|_p=\sup\Big\{\Big(\sum_{i=1}^n\max\Big\{x_i,\frac1n\sum_{j=1}^nx_j\Big\}^p\Big)^{1/p}:x_i\geq0,\sum_{i=1}^n x_i^p=1\Big\}.
\end{equation}

We start by proving the lower bound, for a general $0<p<\infty$. For $k\in V$, we consider $\delta_k$, and for every $0< p<\infty$, we have
$$
\|M_{K_n}\delta_k\|_p=\Big(\sum_{i=1}^n\max\Big\{\delta_k(i),\frac1n\sum_{j=1}^n \delta_k(j)\Big\}^p\Big)^{1/p}=\Big(1+\frac{n-1}{n^p}\Big)^{1/p}.
$$

Since $\|\delta_k\|_p=1$, we get that for every $0< p<\infty$
$$
\|M_{K_n}\|_p\geq \Big(1+\frac{n-1}{n^p}\Big)^{1/p}.
$$

Now, by Lemma \ref{normwithdeltas}, for $0<p\le 1$, we get that
$$
\|M_{K_n}\|_p= \Big(1+\displaystyle\frac{n-1}{n^p}\Big)^{1/p}.
$$

Finally, to prove the upper bound for the case $1<p<\infty$, we use Jensen's inequality in \eqref{normnp}:
$$
\|M_{K_n}\|_p\le\sup\Big\{\Big(\sum_{i=1}^n\max\Big\{x_i^p,\frac1n\Big\}\Big)^{1/p}:x_i\geq0,\sum_{i=1}^n x_i^p=1\Big\}.
$$
Now, if $x_i^p\le 1/n$, for every $1\le i\le n$, then
$$
\|M_{K_n}\|_p\le\sup\Big\{\Big(\sum_{i=1}^n\frac1n\Big)^{1/p}:x_i\geq0,\sum_{i=1}^n x_i^p=1\Big\}=1.
$$
On the other hand, if $x_{i_0}^p>1/n$, for some index $ i_0\in\{1,\dots,n\}$, then
\begin{align*}
\|M_{K_n}\|_p&\le\sup\Big\{\Big(\sum_{\{x_i^p>1/n\}}x_i^p+\sum_{\{x_i^p\le1/n\}}\frac1n\Big)^{1/p}:x_i\geq0,\sum_{i=1}^n x_i^p=1\Big\}\\
&\le \bigg(1+\frac{n-1}{n}\bigg)^{1/p}.
\end{align*}
\end{proof}

It is not an easy task to compute the exact value of $\|M_{K_n}\|_p$ for $p>1$. At least, from Proposition~\ref{estimateKn}, we know that $1\leq \|M_{K_n}\|_p\leq2$, for every $n\in\mathbb N$ and $p>1$.

\begin{remark}
The estimates we have obtained in Proposition~\ref{estimateKn} \textit{(ii)} are not optimal in general. For example, if we consider the case $n=2$, then for every function $f:\{1,2\}\rightarrow\mathbb R^+$, we have that $M_{K_2}f(j)=(f(j)+\Vert f\|_\infty)/2$. Thus, if we assume that $\|f\|_\infty=f(2)$ and set $\alpha=f(1)/f(2)$, then for every $0<p<\infty$,

$$
\frac{\|M_{K_2}f\|_p^p}{\|f\|_p^p}=\frac1{2^p}\frac{(f(1)+f(2))^p+2^pf(2)^p}{f(1)^p+f(2)^p}=\frac1{2^p}\frac{(1+\alpha)^p+2^p}{1+\alpha^p},
$$
and hence,
$$
\|M_{K_2}\|_p=\frac12\bigg(\sup_{0\le\alpha\le1}\frac{(1+\alpha)^p+2^p}{1+\alpha^p}\bigg)^{1/p}.
$$
It is easy to see that, for $1<p<\infty$, this supremum is attained at the unique root $\alpha_p\in(0,1)$ of the equation
$$
(1+\alpha)^{p-1}=\frac{2^p\alpha^{p-1}}{1-\alpha^{p-1}}.
$$
In particular, if $p=2$, then $\alpha_2=\sqrt{5}-2$ and $\|M_{K_2}\Vert_2=(3+\sqrt{5})^{1/2}/2.$ However, from \eqref{ubmkp} we only obtain that $\sqrt{5}/2<\|M_{K_2}\Vert_2 <(3/2)^{1/2}$.
\end{remark}

As we have seen, and contrary to what happens for the case $0<p\le1$, the lower estimate given in \eqref{ubmkp} is not optimal  when $1<p<\infty$. A closer look to the proof of this result shows that this bound  is obtained by evaluating the maximal operator on characteristic functions supported at a singleton (a Kronecker delta).  This can be improved by considering arbitrary  characteristic functions (what is usually called a restricted type estimate):
$$
\|M_{K_n}\|_{p,{\rm rest}}=\max\Bigg\{\frac{\|M_{K_n}\chi_A\|_p}{\|\chi_A\|_p}: A\subset V\Bigg\}.
$$
Clearly, $\|M_{K_n}\|_{p,{\rm rest}}\leq\|M_{K_n}\|_p$. The following result shows  that, for some particular values of $n\ge2$ and $p>1$, we can get a better estimate. Recall that $p'$ denotes the conjugate index to $p$, defined as $1/p+1/p'=1$, and $[x]$ is the integer part of $x$.

\begin{proposition}\label{estresttyp}
Let $n\ge2$ and $p>1$.
\begin{enumerate}[{(i)}]
\item
If $n\leq p'$, then
$$
\|M_{K_n}\|_{p,{\rm rest}}=\bigg(1+\frac{n-1}{n^p}\bigg)^{1/p}.
$$

\item
If $n\leq p$, then
$$
\|M_{K_n}\|_{p,{\rm rest}}=\bigg(1+\frac{(n-1)^{p-1}}{n^p}\bigg)^{1/p}.
$$

\item
If $n>\max\{p,p'\}$, $p\in\mathbb{Q}$ with $p={p_1}/{p_2}$ and $p_1$ divides $n$, then
$$
\|M_{K_n}\|_{p,{\rm rest}}=\bigg(1+\frac{(p-1)^{p-1}}{p^p}\bigg)^{1/p}.
$$

\item
If $n>\max\{p,p'\}$, but $p$ is not of the previous form, and $[n]_p=[n/p']$, then
$$
\|M_{K_n}\|_{p,{\rm rest}}=\bigg(1+\frac{1}{n^p}\max\Big\{\Big(n-[n]_p\Big)[n]_p^{p-1},\Big(n-1-[n]_p\Big)\Big([n]_p+1\Big)^{p-1}\Big\}\bigg)^{1/p}.
$$
\end{enumerate}
In particular, if $n>p'$  we have that
$$
\|M_{K_n}\|_p\ge\|M_{K_n}\|_{p,{\rm rest}}>\Big(1+\frac{n-1}{n^p}\Big)^{1/p}.
$$
\end{proposition}

\begin{proof}
For $A\subset V$, with $|A|=k\leq n$, we have
$$
M_{K_n}\chi_A(j)=\left\{
                   \begin{array}{ccl}
                     1, & \textrm{if }j\in A, \\
                      [.2cm]
                     {k}/{n,} & \textrm{if }j\notin A. \\
                   \end{array}
                 \right.
$$

Therefore,
$$
\|M_{K_n}\chi_A\|_p=\Big(\sum_{j=1}^n M_{K_n}\chi_A(j)^p\Big)^{1/p}=\Big(k+\frac{(n-k)k^p}{n^p}\Big)^{1/p}.
$$
Since $\|\chi_A\|_p=k^{1/p}$, we get
\begin{equation}\label{nres}
\|M_{K_n}\|_{p,{\rm rest}}=\Big(1+\frac{1}{n^p}\max_{1\leq k\leq n-1}(n-k)k^{p-1}\Big)^{1/p}.
\end{equation}
To compute this supremum, let us consider the function $\varphi(x)=(n-x)x^{p-1}$, for $x>0$. It is easy to see that $x=n/p'$ is the critical point of   $\varphi$; that is, $\varphi'(n/p')=0$.
\medskip

\begin{enumerate}[{\it(i)}]
\item
If $n\leq p'$, then   $\varphi$ is a monotone function on $[1,n-1]$ and  the above supremum is attained at the endpoints. This means that
$$
\|M_{K_n}\|_{p,{\rm rest}}=\max\Big\{\Big(1+\frac{n-1}{n^p}\Big)^{1/p}, \Big(1+\frac{(n-1)^{p-1}}{n^p}\Big)^{1/p}\Big\}.
$$
Since $n\leq p'$, then  this maximum is precisely the  first term.
\medskip

\item
If $n\leq p$, then $n/p'\geq n-1$ and, as in the previous case, we get that
$$
\|M_{K_n}\|_{p,{\rm rest}}=\max\Big\{\Big(1+\frac{n-1}{n^p}\Big)^{1/p}, \Big(1+\frac{(n-1)^{p-1}}{n^p}\Big)^{1/p}\Big\},
$$
and now  the maximum agrees with the second term.
\medskip

\item
If $n>\max\{p,p'\}$, $p\in\mathbb{Q}$, with $p={p_1}/{p_2}$ and $p_1$ divides $n$, then the critical point $n/p'$ is an integer between $1$ and $n-1$, so the supremum in \eqref{nres} is attained at this point. Thus,
$$
\|M_{K_n}\|_{p,{\rm rest}}=\Big(1+\frac{(p-1)^{p-1}}{p^p}\Big)^{1/p}.
$$

\item
If $n>\max\{p,p'\}$, but $p$ is not of the previous form, then the critical point $n/p'\in[1,n-1]$, but it is not an integer, so the above supremum is
$$
\|M_{K_n}\|_{p,{\rm rest}}^p=1+\frac{1}{n^p}\max\Big\{\Big(n-[n]_p\Big)[n]_p^{p-1},\Big(n-1-[n]_p\Big)\Big([n]_p+1\Big)^{p-1}\Big\},
$$
which corresponds to the evaluation at the closest integer.
\end{enumerate}

The fact that $\|M_{K_n}\|_{p,{\rm rest}}>\big(1+\frac{n-1}{n^p}\big)^{1/p}$, if $n>p'$, is an easy computation. For example, if $p'<n\le p$, then $p> 2$, which is equivalent to the inequality 
$$
\bigg(1+\frac{(n-1)^{p-1}}{n^p}\bigg)^{1/p}>\bigg(1+\frac{n-1}{n^p}\bigg)^{1/p}.
$$
\end{proof}

\section{Optimal estimates for  $\|M_G\|_p$}\label{optgr}

In this Section we are going to prove our main result, namely that if $0<p\le1$, the norm of $M_G$ is bounded below and above by some optimal constants, and that equality at the endpoints is only obtained for some specific graphs. Throughout we fix $n\in\mathbb{N}$ and $V=\{1,\ldots,n\}$. Let $S_n$ denote the star graph of $n$ vertices; i.e., a graph with one vertex of degree $n-1$ and $n-1$ leafs (vertices of degree 1). It is clear that, on $V$, there are $n$ different (but isomorphic) $n$-star graphs.
\bigbreak

\begin{theorem}\label{knsn}
Let $G$ be a graph with $n$ vertices and $0<p\le1$. Then, the following optimal  estimates hold:
$$
\Big(1+\frac{n-1}{n^p}\Big)^{1/p}\leq \|M_G\|_p\leq \Big(1+\frac{n-1}{2^p}\Big)^{1/p}.
$$
Moreover,  
\begin{enumerate}[{(i)}]
\item
$\|M_G\|_p=\Big(1+\displaystyle\frac{n-1}{n^p}\Big)^{1/p}$    if and only if  $G= K_n$;
\medskip

\item 
$\|M_G\|_p=\displaystyle\Big(1+\frac{n-1}{2^p}\Big)^{1/p}$ if and only if $G\sim S_n$.
\end{enumerate}
\end{theorem}

\begin{proof}
This theorem contains several claims. We will prove each of them separately.
\vskip .3cm

\textit{Claim 1: For every graph $G$ we have $\big(1+\frac{n-1}{n^p}\big)^{1/p}\leq \|M_G\|_p\leq \big(1+\frac{n-1}{2^p}\big)^{1/p}$.}
\medskip

Using \eqref{desphl}  we have that $M_{K_n}f\leq M_G f$. Hence,  Proposition \ref{estimateKn} gives us the lower estimate

$$
\Big(1+\frac{n-1}{n^p}\Big)^{1/p}=\|M_{K_n}\|_p\leq\|M_G\|_p.
$$

For the upper bound, given $i\in V$ we have that
$$
\|M_G\delta_i\|_p^p=M_G\delta_i(i)^p+\sum_{k\in V\backslash \{i\}} M_G\delta_i(k)^p=1+\sum_{k\in V\backslash \{i\}} \bigg(\frac{1}{|B_k|}\sum_{j\in B_k}\delta_i(j)\bigg)^p,
$$
where $B_k$ is a ball in $G$ with center $k$ and a certain radius grater than or equal to 1. Note that for $k\in V\backslash \{i\}$ necessarily $|B_k|\geq2$. Thus, we get
$$
\|M_Gf\|_p^p\leq 1+ \frac{n-1}{2^p}.
$$
Since this holds for each $i\in V$, by Lemma \ref{normwithdeltas} we obtain the upper estimate

$$
\|M_G\|_p\leq \Big(1+\frac{n-1}{2^p}\Big)^{1/p}.
$$
\medskip

\textit{Claim 2: $G= K_n$ if and only if $\|M_G\|_p=\big(1+\frac{n-1}{n^p}\big)^{1/p}$.}
\medskip

By Proposition \ref{estimateKn}, we have that $\|M_{K_n}\|_p=\big(1+\frac{n-1}{n^p}\big)^{1/p}$ and hence it remains to show that any graph $G$ with $n$ vertices, which is not  $K_n$, must necessarily satisfy $\|M_G\|_p>\big(1+\frac{n-1}{n^p}\big)^{1/p}$. To see this, suppose $G\not= K_n$. Then, there exist $i\neq j$ in $V=\{1,\ldots,n\}$ such that $d_G(i,j)>1$. Let us consider the sets

$$
A=B(i,1)=\{k\in V: \,d_G(i,k)\leq1\}\quad\text{and}\quad B=B(j,1)=\{k\in V: \,d_G(j,k)\leq1\}.
$$

Clearly $|A|,|B|\ge 2$. We will analyze two cases:
\begin{itemize}
  \item[(a)] $\min\{|A|,|B|\}\le n/2$.
  \item[(b)] $\min\{|A|,|B|\}> n/2$.
\end{itemize}

In case (a), we may suppose without loss of generality that  $|A|\le n/2$. We pick any $k\in A$ such that $k\neq i$ (i.e., $d_G(i,k)=1$) and define  $\delta_k$ as in \eqref{krodel}. Then, since $M_G\delta_k(l)\ge 1/n$,  for every $l\in V$,
\begin{align*}
\|M_G \delta_k\|_p^p&=\sum_{l=1}^n M_G\delta_k(l)^p\\
&=M_G\delta_k(k)^p+M_G\delta_k(i)^p+\sum_{l\neq i,k} M_G\delta_k(l)^p\\
&\geq 1+\bigg(\frac{1}{|A|}\sum_{m\in A} \delta_k(m)\bigg)^p+\frac{n-2}{n^p}.
\end{align*}
Using the hypotheses ($k\in A$ and $|A|\leq  {n}/{2}$),  we now get
$$
\|M_G\|_p^p\geq\|M_G \delta_k\|_p^p\geq 1+ \bigg(\frac{2}{n}\bigg)^p+\frac{n-2}{n^p}>1+\frac{n-1}{n^p}.
$$
This finishes the proof for (a).

We now consider case (b), in which both $A$ and $B$ have cardinality strictly larger than $n/2$. In particular,  we have that $A\cap B\neq\emptyset$. If we pick $k\in A\cap B$ and consider the function $\delta_k$ as above, then
\begin{align*}
\|M_G \delta_k\|_p^p&=M_G\delta_k(i)^p+M_G\delta_k(j)^p+M_G\delta_k(k)^p+\sum_{l\neq i,j,k} M_G\delta_k(l)^p\\
& \geq \bigg(\frac{1}{|A|}\sum_{m\in A} \delta_k(m)\bigg)^p+\bigg(\frac{1}{|B|}\sum_{m\in B} \delta_k(m)\bigg)^p+1+\frac{n-3}{n^p}.
\end{align*}
Hence, using that $k\in A\cap B$ and $|A|,|B|\leq n-1$, we get
$$
\|M_G\|_p^p\geq\|M_G \delta_k\|_p^p\geq \frac{2}{(n-1)^p}+1+\frac{n-3}{n^p}>1+\frac{n-1}{n^p}.
$$
This proves the claim.
\medskip

\textit{Claim 3: $G\sim S_n$ if and only if $\|M_G\|_p=\big(1+\frac{n-1}{2^p}\big)^{1/p}$.}
\medskip

We  first compute $\|M_{S_n}\|_p$. Let $k\in V$ be the vertex of degree $n-1$ in $S_n$. We have that, for any $f:V\rightarrow \mathbb{R}^+$, with $\|f\|_1=1$,

\begin{equation}\label{msnone}
M_{S_n}f(j)=
              \begin{cases}
                \max\Big\{f(j),\displaystyle\frac1n\Big\},   & \textrm{if }j=k,  \\[.4cm]
                                    \max\Big\{f(j),\displaystyle\frac{f(j)+f(k)}{2},\frac1n \Big\},  & \textrm{if }j\neq k  .
              \end{cases}
\end{equation}

In particular, for $\delta_k$ we get
\begin{equation}\label{sndelta}
\|M_{S_n}\|_p^p\geq\|M_{S_n}\delta_k\|_p^p=\sum_{j=1}^nM_{S_n}\delta_k(j)^p=1+\frac{n-1}{2^p}.
\end{equation}
Since the converse inequality always holds we get that $\|M_{S_n}\|_p=\big(1+\frac{n-1}{2^p}\big)^{1/p}$.
\medskip

Now, suppose that $G$ is not isomorphic to $S_n$, and hence $n\ge3$. Then there exist two different vertices $i,j\in V$ whose degrees satisfy that $d_G(i),d_G(j)>1$. Note that for every function $f:V\rightarrow \mathbb R^+$, with $\|f\|_1\leq1$, either $M_G f(k)=f(k)$ or $M_G f(k)\leq1/({d_k+1})$. Given such $f:V\rightarrow \mathbb R^+$, let
$$
A=\{k\in V: M_G f (k)=f(k)\}.
$$
Then we have
$$
\|M_G f\|_p^p=\sum_{k\in A} M_G f(k)^p+\sum_{k\in V\backslash A} M_G f(k)^p\leq \sum_{k\in A} f(k)^p+\sum_{k\in V\backslash A} \frac1{(d_k+1)^p}.
$$
Now, if both $i,j\in A$, then
$$
\|M_G f\|_p^p\leq 1 +\frac{n-2}{2^p}< 1 +\frac{n-1}{2^p}.
$$
Otherwise, if $i\notin A$, then since $A\neq\emptyset$, we have
$$
\|M_G f\|_p^p\leq1+\frac1{(d_i+1)^p}+\frac{n-2}{2^p}\leq1+\frac1{3^p}+\frac{n-2}{2^p}<1 +\frac{n-1}{2^p}.
$$
Similarly, the same holds if $j\notin A$.  Hence,
$$
\|M_G \|_p^p\le\max\Big\{1 +\frac{n-2}{2^p},1+\frac1{3^p}+\frac{n-2}{2^p}\Big\}<1 +\frac{n-1}{2^p}.
$$
This proves the claim and finishes the proof.
\end{proof}

In view of the optimality of $\| M_{S_n}\|_1$ and Proposition~\ref{maxec}, it is natural to compare the norms for the corresponding maximal operators for $[L_n]$ and $[S_n]$. The following results show the  different behavior  of these two graphs:

\begin{proposition}\label{optlnsn}\
For $n\ge3$, we have
$$
\|M_{[L_n]}\|_1=   \|M_{[S_n]}\|_1=\|M_{S_n}\|_1=\frac{n+1}2.
$$
\end{proposition}

\begin{proof}
We use Lemma~\ref{normwithdeltas} to estimate $\|M_{[L_n]}\|_1$. Given $j,k\in V$, $j\not=k$, we take any linear tree $L$ for which $k$ is a leaf and $j$ is a neighbor of $k$, to get that
$$
\frac12\ge M_{[L_n]}\delta_j(k)\ge M_{L}\delta_j(k)=\frac12.
$$
 Since $M_{[L_n]}\delta_j(j)=1$, then $\|M_{[L_n]}\delta_j\|_1=1+\frac{n-1}2$, which proves that $\|M_{[L_n]}\|_1=\frac{n+1}2$. Let us now calculate $ \|M_{[S_n]}\|_1$: If $f\ge0$,
\begin{align*}
M_{[S_n]}f(j)&=\max
\bigg\{
\max_{\{1\le k\not= j\le n\}}
\Big\{f(j),\frac{f(k)+f(j)}2,\frac1n\Big\},
\max\Big\{f(j),\frac1n\Big\}
\bigg\}
\\[.2cm]
&=\max\Big\{\frac{f(j)+\|f\|_\infty}2,\frac1n\Big\}.
\end{align*}
Let $A=\big\{k:\frac{f(k)+\|f\|_\infty}2\ge\frac1n\big\}$. Then
\begin{align*}
 \|M_{[S_n]}f\|_1&=\sum_{k\in A}\frac{f(k)+\|f\|_\infty}2+\sum_{k\notin A}\frac1n=\sum_{k\in A}\frac{f(k)}2+\frac{\|f\|_\infty}2|A|+\frac1n(n-|A|)\\
 &\le\frac12+|A|\Big(\frac{\|f\|_\infty}2-\frac1n\Big)+1\le\frac12+n\Big(\frac{1}2-\frac1n\Big)+1=\frac{n+1}2.
 \end{align*}
 Thus, $ \|M_{[S_n]}\|_1\le \frac{n+1}2$. On the other hand, Theorem~\ref{knsn} gives us the converse inequality, since  $ \|M_{[S_n]}\|_1\ge  \|M_{S_n}\|_1= \frac{n+1}2$. Therefore,
$$
\frac{n+1}2\ge \|M_{[S_n]}\|_1\ge \|M_{S_n}\|_1=\frac{n+1}2,
$$
which finishes the proof.
\end{proof}

\begin{proposition}
For $n\geq2$ we have
\begin{equation}\label{strplessone}
\|M_{L_n}\|_p\approx\left\{
\begin{array}{ccc}
  \displaystyle\Big(\frac{n^{1-p}-1}{1-p}\Big)^{1/p}, &  & 0<p<1,    \\
  & &   \\
\log n, & & p=1.
\end{array}
\right.
\end{equation}
\end{proposition}

\begin{proof}
Let us enumerate $L_n=\{1,2,\dots,n\}$, where 1 and $n$ are its leafs. We have
$$
M_{L_n}\delta_1(j)=\frac1{2j-1},\quad 1\le j\le [n/2].
$$
Hence,
$$
\|M_{L_n}\|_1\ge \|M_{L_n}\delta_1\|_1\ge\sum_{j=1}^{[n/2]}\frac1{2j-1}\gtrsim \log n.
$$
Conversely, since $\|M_{L_n}\delta_k\|_1=\|M_{L_n}\delta_{n-k+1}\|_1$, then using Lemma~\ref{normwithdeltas},
\begin{align*}
\|M_{L_n}\|_1&=\max_{1\le k\le [n/2]}\|M_{L_n}\delta_k\|_1\\
&\le \max_{1\le k\le [n/2]}\bigg(\sum_{j=1}^{[k/2]}\frac1k+\sum_{j=[k/2]+1}^{[(k+n)/2]}\frac1{2|k-j|+1}+\sum_{j=[(k+n)/2]+1}^{n}\frac1{n-k+1}\bigg)\\
&\lesssim  \max_{1\le k\le [n/2]}(1+\log n)\lesssim \log n.
\end{align*}

Similarly, for $0<p<1$ we have
\begin{equation*}
\|M_{L_n}\|_p\ge \|M_{L_n}\delta_1\|_p\ge\Big(\sum_{j=1}^{[n/2]}\frac1{(2j-1)^p}\Big)^{1/p}\gtrsim \Big(\frac{n^{1-p}-1}{1-p}\Big)^{1/p}.
\end{equation*}
And, as before,
\begin{align*}
&\|M_{L_n}\|_p=\max_{1\le k\le [n/2]}\|M_{L_n}\delta_k\|_p\\
\le& \max_{1\le k\le [n/2]}\bigg(\sum_{j=1}^{[k/2]}\frac1{k^p}+\sum_{j=[k/2]+1}^{[(k+n)/2]}\frac1{(2|k-j|+1)^p}+\sum_{j=[(k+n)/2]+1}^{n}\frac1{(n-k+1)^p}\bigg)^{1/p}\\
\lesssim&  \max_{1\le k\le [n/2]} \bigg( \frac{k^{1-p}}2+\frac{2\big((k-1)^{1-p}-(k-n-1)^{1-p}\big)}{1-p}+\frac{(n-k+1)^{1-p}}2\bigg)^{1/p}\\
\lesssim& \Big(\frac{n^{1-p}-1}{1-p}\Big)^{1/p}.
\end{align*}
\end{proof}

Note that if $n\ge4$, then $\|M_{[L_n]}\|_1>\|M_{L_n}\|_1.$ Indeed, by Theorem~\ref{knsn} and the fact that $L_n\not\sim S_n$, then we have that
$$
\|M_{L_n}\|_1<\|M_{S_n}\|_1= \|M_{[L_n]}\|_1.
$$
Observe also that in \eqref{strplessone}, $\lim_{p\to 1^-}\Big(\frac{n^{1-p}-1}{1-p}\Big)^{1/p}=\log n$.

\medskip

To finish this Section, we complete the information about the strong-type estimates, on the range $1<p<\infty$, for the star graph.

\begin{proposition}
If $1<p<\infty$, then
$$
\bigg(1+\frac{n-1}{2^p}\bigg)^{1/p}\le\Vert M_{S_n}\|_p\le\bigg(\frac{n+5}{2}\bigg)^{1/p},
$$
i.e., $\Vert M_{S_n}\|_p\approx n^{1/p}$.
\end{proposition}

\begin{proof}
Using an easy modification of \eqref{msnone}, it follows that if $f\ge 0$,
\begin{align*}
 \Vert M_{S_n}f\|_p^p&\le\| f\|_p^p+\frac{n}{n^p}\|f\|_1^p+\sum_{j=1}^n\bigg(\frac{f(j)+f(1)}2\bigg)^p\\
 &\le \| f\|_p^p+n^{1-p}n^{p-1} \| f\|_p^p+\frac1{2^p}\sum_{j=1}^n2^{p/p'}\Big(f^p(j)+f^p(1)\Big)\\
 &\le 2\| f\|_p^p+\frac12\Big(\| f\|_p^p+n\| f\|_p^p\Big)=\frac{n+5}2\| f\|_p^p.
\end{align*}
Conversely, using \eqref{sndelta} we can prove that $\Vert M_{S_n}\|_p\ge  \big(1+\frac{n-1}{2^p}\big)^{1/p}$.
\end{proof}

\section{Weak-type estimates}\label{weaktes}

Let $G$ be a connected graph with $n$ vertices, $V=\{1,\dots,n\}$,   $f:V\rightarrow \mathbb R^+$, and let $ \{f^*_j\}_{j=1,\dots,n}$ be the decreasing rearrangement of the sequence $ \{f_j\}_{j=1,\dots,n}$. We now consider weak-type estimates of the form $M_G:\ell^p(V)\rightarrow\ell^{p,\infty}(V)$, $0<p<\infty$, where
$$
\|f\|_{p,\infty}:=\sup_{t>0}t\big|\{j\in V:f_j>t\}\big|^{1/p}.
$$
It is easily seen that also $\|f\|_{p,\infty}=\max_{j\in V}j^{1/p}f^*_j$. For this purpose  we define
$$
\|M_G\|_{p,\infty}=\sup_f\frac{\|M_Gf\|_{p,\infty}}{\|f\|_p}.
$$
It is clear that if $0<p<\infty$, then $\|M_G\|_{p,\infty}\le \|M_G\|_{p}$ and also
\begin{equation}\label{uppboun}
\|M_G\|_{p,\infty}\le n^{1/p}, \quad \text{if } |G|=n.
\end{equation}

\begin{theorem}\label{weakkn}
 If $0<p<\infty$, then
\begin{equation}\label{weakmax}
\|M_{K_n}\|_{p,\infty}=
\begin{cases}
  n^{1/p-1},    & \text{if } 0<p\le1,\\
   1,   & \text{if } p\geq1.
\end{cases}
\end{equation}
In particular, for every connected graph $G$ with $n$ vertices,
$$
\|M_{G}\|_{p,\infty}\ge\begin{cases}
  n^{1/p-1},    & \text{if } 0<p\le1,\\
   1,   & \text{if } p\geq1.
   \end{cases}
$$
\end{theorem}

\begin{proof}
Let $f:V\rightarrow \mathbb R^+$, with $\|f\|_p=1$ (we may assume that $f$ is not a constant function), and let $A(f)=\| f\|_1/n$. Since $M_{K_n}f(j)=\max\{f_j,A(f)\}$, if we define
$$
j(f)=\min\{1\le j\le n-1: f^*_{j+1}< A(f)\le f^*_j\},
$$
then
$$
\big(M_{K_n}f\big)^*(j)=
\begin{cases}
  f^*_j,    & \text{if } 1\le j\le j(f),\\
   A(f),   & \text{if } j(f)<j\le n,
\end{cases}
$$
and
\begin{align}\label{norminfp}
\|M_{K_n}f\|_{p,\infty}&=\max\Big\{\max_{1\le j\le j(f)}j^{1/p}f^*_j,\max_{j(f)<j\le n}j^{1/p}A(f)\Big\}\nonumber\\
&=\max\Big\{\max_{1\le j\le j(f)}j^{1/p}f^*_j,n^{1/p-1}\|f\|_1\Big\}.
\end{align}
If we now take $f=\delta_1$, then $j(f)=1$ and
$$
\|M_{K_n}\delta_1\|_{p,\infty}=\max\Big\{1,n^{1/p-1}\Big\}=
\begin{cases}
  n^{1/p-1},    & \text{if } 0<p\le1,\\
   1,   & \text{if } p\geq1,
\end{cases}
$$
and, hence
$$
\|M_{K_n}\|_{p,\infty}\ge
\begin{cases}
  n^{1/p-1},    & \text{if } 0<p\le1,\\
   1,   & \text{if } p\geq1.
\end{cases}
$$

Conversely, if $0<p\le1$ let us see that $f^*_j\le 1/j$, for every $1\le j\le n$. In fact, if there exists $1\le j_0 \le n$ for which $f^*_{j_0}>1/j_0$, then
$$
1\ge\sum_{j=1}^{j_0}(f^*_j)^p>\sum_{j=1}^{j_0}\frac1{{j_0}^p}=j^{1-p}_0\ge 1,
$$
which is a contradiction. Thus,  $j^{1/p}f^*_j\le j^{1/p-1}\le n^{1/p-1},$ and hence $\|M_{K_n}\|_{p,\infty}\le  n^{1/p-1}$.

Finally, if $1<p<\infty$, we have  that $j^{1/p}f^*_j\le\|f\|_{p,\infty}\le\|f\|_p=1$ and also, using H\"older's inequality, $n^{1/p-1}\|f\|_1\le n^{1/p-1}\|f\|_p\, n^{1/p'}=\|f\|_p=1$, and the result follows from \eqref{norminfp}.

The last part is a consequence of the trivial estimate $M_{K_n}f(j)\le M_{G}f(j)$, for every $j\in V$,  and \eqref{weakmax}.
\end{proof}

\begin{remark}
The fact that $\|M_{K_n}\|_{1,\infty}=1$ also follows from the general theory for ultrametric spaces. Recall that an ultrametric space is a metric space with the stronger inequality
$$
d(x,y)\leq\max\{d(x,z),d(z,y)\}
$$
instead of the triangle inequality. It is clear that $K_n$ is an ultrametric space. In fact, it is the only graph with this property: Indeed, if $G\neq K_n$, there exist two vertices $x,y$ with $d_G(x,y)=r\geq2$. Pick a geodesic path joining $x$ and $y$, and let $z$ be a neighbor of $x$ in that path. It follows that $d_G(x,z)=1$, $d_G(z,y)=r-1$, and  so
$$
d_G(x,y)=r>\max\{d(x,z),d(z,y)\}.
$$
\end{remark}

\begin{proposition}\label{wtpstar} Let $0<p<\infty$. Then,
\begin{equation}\label{normsn}
\max\{n^{1/p}/2,1\} \le\|M_{S_n}\|_{p,\infty}\le n^{1/p}.
\end{equation}

In particular,  $\|M_{S_n}\|_{p,\infty}\approx n^{1/p}$, for every $n\ge1$ and $0<p<\infty$, and also, for every connected graph $G$ with $n$ vertices, $\|M_G\|_{p,\infty}\le 2\|M_{S_n}\|_{p,\infty}$, $0<p<\infty$.
\end{proposition}

\begin{proof}
Assuming that $j=1$ is the vertex of degree $n-1$ in $S_n$, and taking $f=\delta_1$, using \eqref{msnone} we get that
$$
M_{S_n}f(j)=
\begin{cases}
  1,    & \text{if } j=1,\\
   1/2,   & \text{if } 2\le j\le n,
\end{cases}
$$
and hence,
$$
\|M_{S_n}f\|_{p,\infty}=\max\big\{1,\max\{j^{1/p}/2:j=2,\dots,n\}\big\}=
\begin{cases}
  1,    & \text{if } n<2^p,\\
   n^{1/p}/2,   & \text{if } n\ge 2^p,
\end{cases}
$$
which, together with the trivial inequality \eqref{uppboun}, proves  \eqref{normsn}. To finish, both estimates $\|M_{S_n}\|_{p,\infty}\approx n^{1/p}$ and $\|M_G\|_{p,\infty}\le 2\|M_{S_n}\|_{p,\infty}$ are  just a simple remark.
\end{proof}

The case $p=1$ in Proposition~\ref{wtpstar} was previously studied in \cite[Proposition~1.5, Remark~1.2]{NaorTao}.
\medskip

Motivated by the classical weak-type $(1,1)$ bounds for the Hardy-Littlewood operator on $\mathbb R^n$ we will introduce two indices associated to a graph $G$: the dilation and overlapping indices. The dilation index of a graph is related to the so called doubling condition, and measures the growth of the number of vertices in a ball when its radius is tripled.

\begin{definition}
Given a graph $G$ we define its dilation index as
$$
\mathcal{D}(G)=\max\bigg\{\frac{|B(x,3r)|}{|B(x,r)|}:x\in V,\, r\in\mathbb N,\, r\leq \mathrm{diam}(G)\bigg\}.
$$
\end{definition}

\begin{example}
The dilation index of the complete graph of $n$ vertices and the star $S_n$ can be easily computed for $n\in\mathbb N$:
$$
\mathcal{D}(K_n)=1 \qquad\text{and}\qquad
\mathcal{D}(S_n)=\frac{n}{2}.
$$
For the linear tree $L_n$ it is easy to check that $\mathcal D(L_n)< 3$ for all $n\in \mathbb {N}$, and that $\lim_{n\rightarrow\infty}\mathcal{D}(L_n)=3$. For small number of vertices we have: $\mathcal D(L_3)=3/2$, $\mathcal D(L_4)=2$, $\mathcal D(L_5)=2$, $\mathcal D(L_6)=2$, $\mathcal D(L_7)=7/3\dots$
\end{example}

The dilation index can be used to give an elementary version of the Vitali covering lemma \cite{Guz}:

\begin{lemma}\label{vitali}
Let $G$ be a graph with $n$ vertices and $A\subset V$ any set of vertices. If $\{B_j\}_{j\in J}$ is a finite collection of balls covering $A$, then there exists $I\subset J$ such that $B_i\cap B_k=\emptyset$, for $i,k\in I$, and
\begin{equation}\label{dilgrphwt}
|A|\leq \mathcal D(G)\sum_{i\in I}|B_i|.
\end{equation}
\end{lemma}

\begin{proof}
Let $B_{i_1}$ be a ball in  $\{B_j\}_{j\in J}$  with the largest radius;  let $B_{i_2}$ be a ball in  $\{B_j\}_{j\in J\setminus{\{i_1\}}}$, with the largest radius among those which are disjoint from $B_{i_1}$;  let $B_{i_3}$ be a ball in  $\{B_j\}_{j\in J\setminus{\{i_1,i_2\}}}$, with the largest radius among those which are disjoint from $B_{i_1}$ and $B_{i_2}$, and so on. Let $k$ be the index where this process stops, and set $I=\{i_1,\dots,i_k\}$.

That $\{B_i\}_{i\in I}$ are pairwise disjoint is trivial by construction. To prove \eqref{dilgrphwt}, given a ball $B_i=B(x_i,r_i)$ let us consider $\widetilde B_i=B(x_i,3r_i)$. We claim that $A\subset\bigcup_{i\in I} \widetilde B_i$. Indeed, otherwise there is a vertex $v\in A\backslash \bigcup_{i\in I} \widetilde B_i$, and since $A\subset\bigcup_{j \in J} B_j$ we have that $v\in B_{j_0}=B_{j_0}({x_{j_0},r_{j_0})}$, for some $j_0\in J\backslash I$. Since the ball $B_{j_0}$ has not been chosen, there exists  $i\in I$ such that $B_{j_0}\cap B_{i}\neq\emptyset$ and $r_i\ge r_{j_0}$. Finally, if we take $u\in B_{j_0}\cap B_{i}$, then
$$
d_G(v,x_i)\le d_G(v,x_{j_0})+d_G(x_{j_0},u)+d_G(u,x_i)\le r_{j_0}+r_{j_0}+r_i\le 3 r_i,
$$
and hence $v\in \widetilde B_i$, which is a contradiction.

Therefore, $A\subset\bigcup_{i\in I} \widetilde B_i$ and we have
$$
|A|\leq \sum_{i\in I}|\widetilde B_i|\leq \mathcal D(G)\sum_{i\in I}|B_i|.
$$
\end{proof}

Another useful quantity for weak-type $(1,1)$ estimates of the maximal operator is the overlapping index of a graph, which represents the smallest number of balls that necessarily overlap in a covering of the graph:

\begin{definition}
Given a graph $G$ we define its overlapping index as
\begin{align*}
\mathcal{O}(G)=\min\bigg\{r\in\mathbb N: \ & \forall \{B_j\}_{j\in J},\,  B_j\text{ a ball in } G,\,\exists I\subset J,\\
&\qquad  \bigcup_{j\in J}B_j=\bigcup_{i\in I}B_i \text{ and }\sum_{i\in I}\chi_{B_i}\leq r\bigg\}.
\end{align*}
\end{definition}

\begin{example}
The overlapping index of the following families of graphs can be computed easily:

\begin{align*}
\mathcal{O}(K_n)&=1,\, \forall n\in\mathbb N; \quad
&\mathcal{O}(S_n)&= n-1,\,\forall n\geq2;\\
\mathcal{O}(L_n)&=\left\{
\begin{array}{ccc}
 1 &   & n\leq2,  \\
 2 &   & n\geq 3;
\end{array}
\right.
\quad
&\mathcal{O}(C_n)&=\left\{
\begin{array}{ccc}
 1 &   & n\leq3, \\
 2 &   & n\geq 4.
\end{array}
\right.
\end{align*}
\end{example}

The dilation and overlapping indices provide an upper bound for the weak-type $(1,1)$ norm of the maximal operator of a graph:

\begin{theorem}\label{indices}
Given a graph $G$, we have
$$
\|M_G\|_{1,\infty}\leq\min\big\{\mathcal{D}(G),\mathcal{O}(G)\big\}.
$$
\end{theorem}

\begin{proof}
The proof follows the same kind of arguments used for estimating in $\mathbb R^n$ the weak-type boundedness of the classical centered Hardy-Littlewood maximal operator $M$. Given $f:V\rightarrow\mathbb R$ and $t>0$, let
$$
A_t=\{1\le j\le n: M_{G}f(j)>t\}.
$$
For each $j\in A_t$, take a ball $B_j\subset G$ centered at $j$, satisfying that
$$
\sum_{k\in B_j}|f(k)|>t|B_j|.
$$

On the one hand, by Lemma \ref{vitali}, there exists $I\subset A_t$ such that $(B_i)_{i\in I}$ are pairwise disjoint and
$$
|A_t|\leq \mathcal D(G)\sum_{i\in I} |B_i|.
$$
Therefore, we get
$$
|A_t|\leq \mathcal D(G)\sum_{i\in I} |B_i|\leq\mathcal D(G)\sum_{i\in I}\frac1t \sum_{k\in B_j}|f(k)|\leq \frac{\mathcal D(G)}{t}\|f\|_1.
$$
Thus, we have $\|M_{G}\|_{1,\infty}\le \mathcal{D}(G)$.

On the other hand, using the definition of the overlapping index, we can also select $L\subset A_t$ such that $A_t\subset \cup_{l\in L} B_l$ and $\sum_{l\in L}\chi_{B_l}(k)\le \mathcal{O}(G)$, $k=1,\dots,n$. Hence,
$$
|A_t|\le \sum_{l\in L}|B_l|<\frac1t\sum_{l\in L}\sum_{k\in B_l}|f(k)|\le \frac {\mathcal{O}(G)}{t}\|f\|_1,
$$
which shows that $\|M_{G}\|_{1,\infty}\le \mathcal{O}(G)$.
\end{proof}

To illustrate these results, we will consider now the particular case of the linear tree $L_n$. We will see that for this graph, the behavior of the maximal operator is similar to what happens in the euclidean setting $\mathbb R$. First we prove an interpolation result in $L^{p,\infty}(\mu)$, for a general measure $\mu$.

 \begin{lemma}\label{interpol}
 If $T$ is a sublinear operator, of weak-type (1,1) with constant $C_1$ and bounded in $L^\infty$ with constant $C_\infty$, then $T:L^{p}\rightarrow L^{p,\infty}$ is bounded, for $1< p<\infty$, and
 $$
 \| Tf\|_{p,\infty}\le(p')^{1/p'}p^{1/p}C_1^{1/p}C_\infty^{1/p'}\|f\|_p.
 $$
 \end{lemma}

 \begin{proof}
 Fix $t>0$, $0<\lambda<1$, and set $r=(1-\lambda)t/C_\infty$. For $f\in L^p$, write $f=f\chi_{\{|f|>r\}}+f\chi_{\{|f|\le r\}}=f_1+f_2$. Then,
\begin{align*}
 \mu(\{|Tf|>t\})&\le  \mu(\{|Tf_1|>\lambda t\})+\mu(\{|Tf_2|>(1-\lambda)t\})\\
 &\le \frac{C_1}{\lambda t}\|f_1\|_1+ \mu(\{C_\infty\|f_2\|_\infty>(1-\lambda)t\})\\
 &= \frac{C_1}{\lambda t}\int_{\{|f|>(1-\lambda)t/C_\infty\}}|f|d\mu\\
 &\le \frac{C_1}{\lambda t}\bigg(\frac{(1-\lambda)t}{C_\infty}\bigg)^{1-p} \int_{\{|f|>(1-\lambda)t/C_\infty\}}|f|^pd\mu\\
&\le \frac{C_1}{C_\infty^{1-p}}\frac{(1-\lambda)^{1-p}}{\lambda}t^{-p}\|f\|_p^p.
\end{align*}
Hence,
$$
\|Tf\|_{p,\infty}\le \inf_{0<\lambda <1}\frac{1}{(1-\lambda)^{1/p'}\lambda^{1/p}}C_1^{1/p}C_\infty^{1/p'}\|f\|_p=(p')^{1/p'}p^{1/p}C_1^{1/p}C_\infty^{1/p'}\|f\|_p.
$$
 \end{proof}

\begin{proposition}  If $1\le p<\infty$, then
\begin{equation*}
1 \le\|M_{L_n}\|_{p,\infty}\le (p')^{1/p'}(2p)^{1/p}\le 3.
\end{equation*}
\end{proposition}

\begin{proof}
Since $\mathcal{O}(L_n)\leq2$, by Theorem \ref{indices} we have that $\|M_{L_n}\|_{1,\infty}\leq2$. Since $\|M_{L_n}\|_\infty=1$ and using Lemma~\ref{interpol}, we finally obtain the result (notice that $(p')^{1/p'}(2p)^{1/p}\le 3$ is a trivial estimate). The fact that $\|M_{L_n}\|_{p,\infty}\ge 1$ follows easily since $(M_{L_n}\delta_1)^*(1)=1$ and hence $\|M_{L_n}\delta_1\|_{p,\infty}\ge 1.$
\end{proof}

\begin{proposition}
If $\{G_n\}_{n\in\mathbb N}$ is a family of graphs such that  $G_n$ has $n$ vertices and $\|M_{G_n}\|_{1,\infty}\approx1$, uniformly in $n$,  then, for every $0<p<1$ and $n\in\mathbb N$, $\|M_{G_n}\|_{p,\infty}\approx n^{1/p-1}$, uniformly in $n$ and $p$. In particular $\|M_{L_n}\|_{p,\infty}\approx n^{1/p-1}$, $0<p<1$.
\end{proposition}

\begin{proof}
The inequality $\|M_{G_n}\|_{p,\infty}\ge n^{1/p-1}$ follows, as before,  from  Theorem~\ref{weakkn}. Now, if $0<p<1$ and $\|f\|_p=1$, then $\|f\|_1\le1$ and
\begin{align*}
j^{1/p}(M_{G_n}f)^*(j)&=j^{1/p-1}j(M_{G_n}f)^*(j)\le j^{1/p-1}\|M_{G_n}f\|_{1,\infty}\\
&\lesssim  j^{1/p-1}\|f\|_1\le  n^{1/p-1}.
\end{align*}
Thus, $\|M_{G_n}\|_{p,\infty}\lesssim n^{1/p-1}$.
\end{proof}

\begin{proposition}\label{lnasin}
For the linear graph $L_n$, we have that $\lim_{n\to\infty}\|M_{L_n}\|_{1,\infty}=2$.
\end{proposition}

\begin{proof}
As we have already seen, we have that $\|M_{L_n}\|_{1,\infty}\leq2$. For the converse inequality, let us assume for simplicity that $n=2k+1$. Then:
$$
M_{L_n}\delta_k(j)=\left\{
                     \begin{array}{cc}
                       \displaystyle \frac1k, & j\leq\displaystyle\Big[\frac{k}{2}\Big]\  \mathrm{ or }\ j>k+\Big[\frac{k+1}{2}\Big],  \\[.5cm]
                      \displaystyle \frac{1}{2|k-j|+1}, & \displaystyle\Big[\frac{k}{2}\Big]<j\leq k+ \Big[\frac{k+1}{2}\Big]. \\
                     \end{array}
                   \right.
$$
Thus, we have
$$
\|M_{L_n}\|_{1,\infty}\geq\|M_{L_n}\delta_k\|_{1,\infty}\geq\frac{n}{k}=\frac{2n}{n-1},
$$
which tends to $2$, as $n\rightarrow\infty$.
\end{proof}

\begin{remark}\label{fertit}
Observe that  $\lim_{n\to\infty}\|M_{L_n}\|_{1,\infty}=2=\|\mathcal M\|_{1,\infty}$, where $\mathcal M$ is the non-centered maximal function in $\mathbb R$. This should be compared to the discretization result proved in \cite[Theorem~3]{TrSo}, namely if $M$ is centered Hardy-Littlewood maximal operator in $\mathbb R$ and we consider the discrete measures 
$$
\mathcal D=\bigg\{\mu=\sum_{k=1}^N\delta_{a_k}: a_k\in\mathbb R,\  a_{k+1}=a_k+H,\ H\text { fixed, } N\in\mathbb N\bigg\},
$$
then 
$$
\sup_{\mu\in\mathcal D}\frac{\|M\mu\|_{1,\infty}}{\|\mu\|}=\frac32.
$$
\end{remark}

\end{document}